\documentclass[12pt]{article}
\usepackage{e-jc}

\usepackage{amsthm,amsmath,amssymb}

\usepackage{graphicx}

\usepackage[colorlinks=true,citecolor=black,linkcolor=black,urlcolor=blue]{hyperref}

\theoremstyle{plain}
\newtheorem{theorem}{Theorem}

\newtheorem{corollary}[theorem]{Corollary}
\newtheorem{proposition}[theorem]{Proposition}

\theoremstyle{definition}

\newtheorem{question}[theorem]{Question}

\theoremstyle{remark}

\title{\bf New upper bound for sums of dilates}

\author{
Albert Bush \qquad  Yi Zhao\thanks{Partially supported by NSF grant DMS-1400073.}\\
\small Department of Mathematics and Statistics\\[-0.8ex]
\small Georgia State University\\[-0.8ex]
\small Atlanta, GA 30303\\
\small\tt albertbush@gmail.com \qquad yzhao6@gsu.edu
}

\date{\dateline{Nov 2, 2016}{Aug 16, 2017}}

\begin{document}

\maketitle


\begin{abstract}
For $\lambda \in \mathbb{Z}$, let $\lambda \cdot A = \{ \lambda a : a \in A\}$.  
Suppose $r, h\in \mathbb{Z}$ are sufficiently large and comparable to each other. 
We prove that if $|A+A| \le K |A|$ and $\lambda_1, \ldots, \lambda_h \le 2^r$, then 
\[ |\lambda_1 \cdot A + \ldots + \lambda_h \cdot A | \le K^{ 7 rh /\ln (r+h) } |A|. \]
This improves upon a result of Bukh who shows that
\[ |\lambda_1 \cdot A + \ldots + \lambda_h \cdot A | \le K^{O(rh)} |A|. \]
Our main technique is to combine Bukh's idea of considering the binary expansion of $\lambda_i$ with a result on biclique decompositions of bipartite graphs.

\end{abstract}

\section{Introduction}
Let $A$ and $B$ be nonempty subsets of an abelian group, and define the \emph{sumset} of $A$ and $B$ and the \emph{$h$-fold sumset} of $A$ as 
\begin{align*}A+B := \{ a + b: a \in A, b \in B\} \text{\hspace{3px} and \hspace{3px}} hA  := \{ a_1 + \ldots + a_h: a_i \in A\},
\end{align*}
respectively.  When the set $A$ is implicitly understood, we will reserve the letter $K$ to denote the doubling constant of $A$; that is, $K := |A+A|/|A|$.  A classical result of Pl\"unnecke bounds the cardinality of $hA$ in terms of $K$ and $|A|$.
\begin{theorem}[Pl\"unnecke's inequality \cite{plu70}]\label{plunnecke}
For any set $A$ and for any nonnegative integers $\ell$ and $m$, if $|A+A| = K |A|$, then
\[ |\ell A - m A| \le K^{\ell + m} |A|. \]
\end{theorem}
\noindent See the survey of Ruzsa \cite{ruz09} for variations, generalizations, and a graph theoretic proof of Theorem \ref{plunnecke}; see Petridis \cite{pet12} for a new inductive proof.

Given $\lambda \in \mathbb{Z}$, define a \emph{dilate} of $A$ as 
\begin{align*}\lambda \cdot A := \{ \lambda a: a \in A\}.
\end{align*}
Suppose $\lambda_1 , \ldots , \lambda_h$ are nonzero integers. Since $\lambda_i \cdot A \subseteq \lambda_i A$, one can apply Theorem \ref{plunnecke} to conclude that
\begin{align*}|\lambda_1 \cdot A + \ldots + \lambda_h \cdot A| \le K^{\sum_i |\lambda_i|}|A|. 
\end{align*}
Bukh \cite{buk08} significantly improved this by considering the binary expansion of $\lambda_i$ and using Ruzsa's covering lemma and triangle inequality.   
\begin{theorem}[Bukh \cite{buk08}]\label{buk08}
For any set $A$, if $\lambda_1 , \ldots , \lambda_h \in \mathbb{Z}\setminus \{0\}$ and $|A+A| = K|A|$, then
\[ | \lambda_1 \cdot A + \ldots + \lambda_h \cdot A|  \le K^{7 + 12 \sum_{i=1}^h \log_2 (1+ |\lambda_i|)}|A|. 
\]
\end{theorem}

If $|\lambda_i| \le 2^r$ for all $i$, 
then Theorem \ref{buk08} yields that 
\begin{align}\label{logbound}
| \lambda_1 \cdot A + \ldots + \lambda_h \cdot A| \le K^{O(rh)}|A|. 
\end{align}
In this paper we prove a bound 
that improves \eqref{logbound} when $r$ and $h$ are sufficiently large and comparable to each other.
Throughout the paper $ln$ stands for the natural logarithm.
\begin{theorem}\label{mainthm}
Suppose $r, h \in \mathbb{Z}$ are sufficiently large and 
\begin{equation}
\label{eq:mn}
\min \{r+1, h\} \ge 10 \left(\ln \max\{r+1, h\} \right)^2. 
\end{equation}
Given a set $A$ and nonzero integers $\lambda_1, \ldots, \lambda_h$ such that $|\lambda_i| \le 2^r$, if $|A+A| = K|A|$, then
\begin{equation}
\label{eq:main}
 | \lambda_1 \cdot A + \ldots + \lambda_h \cdot A| \le K^{7 rh/  \ln (r+h)}   |A|.
\end{equation}
\end{theorem}
The proof of Theorem \ref{mainthm} relies on Theorem \ref{buk08} as well as a result of Tuza~\cite{Tuza84} on decomposing bipartite graphs into bicliques (complete bipartite subgraphs). The key idea is to connect Bukh's technique of considering the binary expansion of $\lambda_i$ to the graph decomposition problem that allows us to efficiently group certain powers of 2.

We remark here that in all of the above theorems, the condition $|A+A| = K|A|$ can be replaced with $|A-A| = K|A|$ with no change to the conclusion.  
It is likely that Theorem \ref{mainthm} is not best possible -- we discuss this in the last section. 



\section{Basic Tools}
We need the following analogue of Ruzsa's triangle inequality, see \cite[Theorem 1.8.7]{ruz09}.
\begin{theorem}[Ruzsa \cite{ruz09}] 
\label{ruzsatriangle}
For any sets $X, Y,$ and $Z$,
\[ | X + Y| \le \frac{|X+Z||Z+Y|}{|Z|}. \]
\end{theorem}
A useful corollary of Theorem~\ref{ruzsatriangle} is as follows.
\begin{corollary}
\label{easyobs}
For any sets $A$ and $B$, if $p_1$ and $p_2$ are nonnegative integers and $|A+A| \le K|A|$, then
\[ |B + p_1 A - p_2 A| \le K^{p_1 + p_2 + 1}|B + A|. \] 
\end{corollary}
\begin{proof}
Apply Theorem \ref{ruzsatriangle} with $X = B$, $Y =  p_1 A - p_2 A$, and $Z = A$, then apply Pl\"unnecke's inequality (Theorem \ref{plunnecke}).
\end{proof}
We can use Corollary~\ref{easyobs} to prove the following proposition that we will use in the proof of Theorem \ref{mainthm}.
\begin{proposition}\label{easyobsgen}
If $k_1, \ell_1, \ldots , k_q , \ell_q$ are nonnegative integers, $K> 0$, and $A_1, \ldots , A_q, C$ are sets such that $|A_i + A_i| \le K|A_i|$, then
\begin{equation}\label{withc} |C + k_1 A_1 - \ell_1 A_1 + \ldots + k_q A_q - \ell_q A_q| \le |C + A_1 + \ldots + A_q| \cdot K^{q + \sum_{i=1}^q (k_i + \ell_i)}.
\end{equation}
In particular,
\begin{equation}\label{withoutc}
|k_1 A_1 - \ell_1 A_1 + \ldots + k_q A_q - \ell_q A_q| \le |A_1 + \ldots + A_q| \cdot K^{q + \sum_{i=1}^q (k_i + \ell_i)}.
\end{equation}
\end{proposition}
\begin{proof}
\eqref{withoutc} follows from \eqref{withc} by taking $C$ to be a set with a single element so it suffices to prove \eqref{withc}. We proceed by induction on $q$.  The case $q = 1$ follows from Corollary~\ref{easyobs} immediately.  When $q > 1$, suppose the statement holds for any positive integer less than $q$.  Applying Corollary~\ref{easyobs} with $B = C + k_1 A_1 - \ell_1 A_1 + \ldots + k_{q-1} A_{q-1} - \ell_{q-1} A_{q-1}$ and $A = A_q$, we obtain that
\begin{align}\label{stage1} & \hspace{5px} | C + k_1 A_1 - \ell_1 A_1 + \ldots + k_q A_q - \ell_q A_q| \nonumber
\\  \le & \hspace{5px} K^{k_q + \ell_q + 1} |C + k_1 A_1 -\ell_1 A_1 + \ldots + k_{q-1} A_{q-1} - \ell_{q-1} A_{q-1} + A_q|.
\end{align}
Now, let $C' = C + A_q$ and apply the induction hypothesis to conclude that
\begin{align}\label{stage2} & \hspace{5px} |C' + k_1 A_1 - \ell_1 A_1  + \ldots + k_{q-1} A_{q-1} - \ell_{q-1} A_{q-1}| \nonumber
\\  \le & \hspace{5px}|C' + A_1 + \ldots + A_{q-1}| \cdot K^{q-1 + \sum_{i=1}^{q-1} (k_i + \ell_i)}  
\end{align}
Combining \eqref{stage1} with \eqref{stage2} gives the desired inequality:
\begin{equation} |C + k_1 A_1 - \ell_1 A_1 + \ldots + k_q A_q - \ell_q A_q| \le |C + A_1 + \ldots + A_q| \cdot K^{q + \sum_{i=1}^q (k_i + \ell_i)}. \qedhere \nonumber
\end{equation}
\end{proof}

\section{Proof of Theorem \ref{mainthm}}
Given $\lambda_1, \ldots , \lambda_h \in \mathbb{Z} \setminus \{0\}$, 
we define 
\begin{align}\label{rands} 
r := \max_i  \lfloor \log_2 |\lambda_i| \rfloor 
\end{align}
and write the binary expansion of $\lambda_i$ as 
\begin{align}\label{lambda}\lambda_i = \epsilon_i \sum_{j=0}^r \lambda_{i,j} 2^j \text{, where $\lambda_{i,j} \in \{0,1\}$ and $\epsilon_i \in \{-1, 1\}$}.
\end{align}
Bukh's proof of Theorem \ref{buk08} actually gives the following stronger statement. 
\begin{theorem}[\cite{buk08}] 
\label{binbound}
If $\lambda_1, \ldots , \lambda_h \in \mathbb{Z} \setminus \{0\}$ and $|A+A| = K|A|$, then
\begin{align*}| \lambda_1 \cdot A + \ldots + \lambda_h \cdot A| \le K^{7 + 10r + 2 \sum_i \sum_j \lambda_{i,j} }|A|.
\end{align*}
\end{theorem} 
In his proof of Theorem \ref{buk08}, the first step is to observe that
\[
\lambda_1 \cdot A + \ldots + \lambda_h \cdot A \subseteq \sum_{j=0}^r (\lambda_{1,j}  2^j) \cdot A + \ldots + \sum_{j=0}^r (\lambda_{h,j} 2^j) \cdot A.
\]
In our proof, we also consider the binary expansion of $\lambda_i$, but we do the above step more efficiently by first grouping together $\lambda_i$ that have shared binary digits.  In order to do this systematically, we view the problem as a graph theoretic problem and apply the following result of Tuza~\cite{Tuza84}.

\begin{theorem}[Tuza~\cite{Tuza84}]\label{ces}
There exists $n_0$ such that the following holds for any integers $m\ge n \ge n_0$ such that $n\ge 10 (\ln m) ^2 $. Every bipartite graph $G$ on two parts of size $m$ and $n$ can be decomposed into edge-disjoint complete bipartite subgraphs $H_1, \ldots , H_q$ such that $E(G) = \cup_{i} E(H_i)$ and 
\begin{equation}
\label{eq:Tuza}
\sum_{i=1}^q |V(H_i)| \le \frac{3mn}{\ln m}.
\end{equation}
\end{theorem}

Tuza stated this result \cite[Theorem 4]{Tuza84} for the \emph{covers} of $G$, where a \emph{cover} of $G$ is a collection of subgraphs of $G$ such that every edge of $G$ is contained in at least one of these subgraphs. However, the cover provided in his proof is indeed a decomposition. Furthermore, the assumption $n\ge 10 (\ln m) ^2 $ was not stated in \cite[Theorem 4]{Tuza84} but such kind of assumption is needed.\footnote{In his proof, copies of $K_{q, q}$ were repeatedly removed from $G$, where $q= \lfloor \ln m/ \ln j\rfloor$ for $2\le j\le (\ln m)\ln\ln m$. By a well-known bound on the Zarankiewicz problem, every bipartite graph $G$ with parts of size $m$ and $n$ contains a copy of $K_{q, q}$ if $|E(G| \ge (q-1)^{1/q} (n-q+1) m^{1-1/q} + (q-1)m$. A simplified bound $|E(G| \ge (1+ o(1)) n m^{1-1/q}$ was used in \cite{Tuza84} but it requires that $q m^{1/q} = o(n)$.}
Indeed, \eqref{eq:Tuza} becomes false when $n= o(\ln m)$ because $\sum_{i=1}^q |V(H_i)| \ge m+ n$ for any cover $H_1, \ldots , H_q$ of $G$ if $G$ has no isolated vertices.

Note that \eqref{eq:Tuza} is tight up to a constant factor. Indeed, Tuza~\cite{Tuza84} provided a bipartite graph $G$ with two parts of size $n\le m$ such that every biclique \emph{cover} $H_1, \ldots , H_q$ of $G$ satisfies $\sum_{i=1}^q |V(H_i)| \ge mn/ (e^2\ln m)$, where $e= 2.718...$.

\begin{proof}[Proof of Theorem \ref{mainthm}]
Let $r, h\in \mathbb{Z}$ be sufficiently large and satisfy \eqref{eq:mn}.
Given nonzero integers $\lambda_1, \ldots , \lambda_h$, define $r$ and $\lambda_{i,j}$ as in \eqref{rands} and \eqref{lambda}.  We define a bipartite graph $G = (X,Y,E)$ as follows: let $X = \{ \lambda_1, \ldots , \lambda_h\}$, $Y = \{ 2^0, \ldots , 2^r \}$, and $E = \{ (\lambda_i, 2^j): \lambda_{i,j} = 1 \}$.  In other words, $\lambda_i$ is connected to the powers of 2 that are present in its binary expansion.

We apply Theorem \ref{ces} to $G$ and obtain a biclique decomposition $H_1, \ldots , H_q$ of $G$. 
Assume $H_i := (X_i, Y_i, E_i)$ where $X_i \subseteq X$, $Y_i \subseteq Y$. We have $E_i = \{ (u,v): u\in X_i, v\in Y_i\}$ and
\begin{align}\label{concl} 
\sum_{i=1}^q \left(|X_i| + |Y_i| \right)\le \frac{ 3 (r+1)h }{\ln \max\{r+1, h\} }. 
\end{align}

Now, we connect this biclique decomposition to the sum of dilates $\lambda_1 \cdot A + \ldots + \lambda_h \cdot A$.
Since the elements of $X$ and $Y$ are integers, we can perform arithmetic operations with them.  For $j=1, \ldots , q$, let 
\[ \gamma_j := \sum_{y \in Y_j} y, \]
and since $\mathcal{H}$ is a biclique decomposition, for $i = 1, \ldots , h$,  we have
\[ \lambda_i = \epsilon_i \sum_{j: \lambda_i \in X_j} \gamma_j .\]
Applying the above to each $\lambda_i$ along with the fact that $B + (\alpha + \beta) \cdot A \subseteq B + \alpha \cdot A + \beta \cdot A$ results in
\begin{align}\label{reversesummation} 
\lambda_1 \cdot A + \ldots + \lambda_h \cdot A \subseteq \epsilon_1 \sum_{j: \lambda_1 \in X_j} (\gamma_j \cdot A) + \ldots + \epsilon_h \sum_{j: \lambda_h \in X_j} (\gamma_j \cdot A).
\end{align}
 Let $k_j := | \{ \lambda_i \in X_j : \lambda_i > 0 \}|$, $\ell_j := |\{ \lambda_i \in X_j: \lambda_i < 0\}|$, and note that $k_j + \ell_j = |X_j|$.  By regrouping the terms in \eqref{reversesummation}, we have 
\begin{align*}
& \hspace{5px}\epsilon_1 \sum_{j: \lambda_1 \in X_j} \gamma_j \cdot A + \ldots + \epsilon_h \sum_{j: \lambda_h \in X_j} \gamma_j \cdot A 
\\ = & \hspace{5px} k_1 ( \gamma_1 \cdot A ) - \ell_1 (\gamma_1 \cdot A) + \ldots + k_q ( \gamma_q \cdot A) - \ell_q (\gamma_q \cdot A) .
\end{align*}
Since $|\gamma_j \cdot A + \gamma_j \cdot A| = |A+A| \le K|A| = K| \gamma_j \cdot A |$, we can apply Proposition~\ref{easyobsgen} to conclude that
\begin{align}\label{stp1} 
\nonumber & \hspace{5px}|k_1 ( \gamma_1 \cdot A ) - \ell_1 (\gamma_1 \cdot A) + \ldots + k_q ( \gamma_q \cdot A) - \ell_q (\gamma_q \cdot A)| \\ 
\le & \hspace{5px}|\gamma_1 \cdot A + \ldots + \gamma_q \cdot A| \cdot K^{q + \sum_{i=1}^q k_i + \ell_i}
\le |\gamma_1 \cdot A + \ldots + \gamma_q \cdot A| \cdot K^{2\sum_{i=1}^q |X_i|}.
\end{align}
For $1\le i\le q$ and $0\le j\le r$, let $\gamma_{i,j} = 1$ if $2^j$ is in the binary expansion of $\gamma_i$ and $0$ otherwise.  Observe that 
\[ \max_j \lfloor \log_2 \gamma_j \rfloor \le \max_i \lfloor \log_2 |\lambda_i| \rfloor = r \text{\hspace{5px} and \hspace{5px}} \sum_{j=0}^r \gamma_{i,j} = |Y_i|. \]
Hence, by Theorem \ref{binbound},
\begin{align}\label{stp2} |\gamma_1 \cdot A + \ldots + \gamma_q \cdot A| \le K^{7 + 10r + 2 \sum_{i=1}^q \sum_{j=0}^r \gamma_{i,j}} |A| =  K^{7 + 10r + 2 \sum_{i=1}^q |Y_i|} |A|. 
\end{align}
Combining \eqref{stp1} and \eqref{stp2} with \eqref{concl} results in
\[ | \lambda_1 \cdot A + \ldots + \lambda_h \cdot A| \le K^{7 + 10r + 2 \sum_{i=1}^q (|X_i|+|Y_i|)} |A| \le K^{7 + 10r + \frac{6(r+1)h}{\ln \max\{ r+1, h\} } } |A|.
 \]
We have $7+10r = o( (r+1)h/ \ln \max\{ r+1, h \} )$ because of \eqref{eq:mn} and the assumption that 
$r, h$ are sufficiently large.
Together with 
\[
\ln \max\{r+1, h\}\ge \ln \frac{r+1+h}2  \ge (1- o(1)) \ln (r+h), 
\]
this implies that
$ | \lambda_1 \cdot A + \ldots + \lambda_h \cdot A| \le K^{ 7rh / \ln (r+h) } |A|$, as desired.
\end{proof}


\section{Concluding Remarks}
Instead of Theorem~\ref{ces}, in an earlier version of the paper we applied a result of Chung, Erd\H{o}s, and Spencer \cite{chuerdspe83}, which states that every graph on $n$ vertices has a biclique decomposition $H_1, \ldots , H_q$ such that $\sum_{i=1}^q |V(H_i)| \le (1+o(1)) n^2/ (2\ln n)$.  Instead of \eqref{eq:main}, we obtained that 
\[
|\lambda_1 \cdot A + \ldots + \lambda_h \cdot A | \le K^{O((r+h)^2/\ln (r+h) ) } |A|.
\] 
This bound is equivalent to \eqref{eq:main} when $r=\Theta(h)$ but weaker than Bukh's bound \eqref{logbound}
when $r$ and $h$ are not close to each other. 

Although the assumption \eqref{eq:mn} may not be optimal, Theorem~\ref{mainthm} is not true without any assumption on $r$ and $h$. For example, when $r$ is large and $h=o(\ln r)$, \eqref{eq:main} becomes $|\lambda_1 \cdot A + \ldots + \lambda_h \cdot A | \le K^{ o(r ) } |A|$, which is false when $A= \{1, \dots, n\}$.

If each of $\lambda_1, \dots, \lambda_r$ has $O(1)$ digits in its binary expansion, then Theorem~\ref{buk08} yields that
$|\lambda_1 \cdot A + \ldots + \lambda_r \cdot A| \le K^{O(r)}|A|$. Bukh asked if this bound holds whenever $\lambda_i \le 2^r$: 
\begin{question}[Bukh \cite{buk08}]\label{bukhconj}
For any set $A$ and for any $\lambda_1 , \ldots , \lambda_r \in \mathbb{Z} \setminus \{0\}$, if $|A+A| = K|A|$ and $0 < \lambda_i \le 2^r$, then 
\[ | \lambda_1 \cdot A + \ldots + \lambda_r \cdot A| \le K^{O(r)}|A|. \]
\end{question}  
\noindent In light of Question~\ref{bukhconj}, one can view Theorem~\ref{mainthm} as providing modest progress by proving a subquadratic bound of quality $O(r^2/\ln r)$ whereas Theorem~\ref{buk08} shows that the exponent is $O(r^2)$.

Generalized arithmetic progressions give supporting evidence for Question~\ref{bukhconj}.  A \emph{generalized arithmetic progression} $P$ is a set of the form
\[ P := \{ d + x_1 d_1 + \ldots + x_k d_k: 0 \le x_i < L_i \}. \]
Moreover, $P$ is said to be \emph{proper} if $|P| = L_1 \cdot \ldots \cdot L_k$.  One can calculate that if $P$ is proper, then
\[ |P+P| \le (2L_i - 1)^k \le 2^k |P| =: K |P|. \]
Additionally, one can calculate that for any $\lambda_1, \ldots , \lambda_h \in \mathbb{Z}^+$, if $\lambda_i \le 2^r$ then
\[ |\lambda_1 \cdot P + \ldots + \lambda_h \cdot P| \le (\lambda_1 + \ldots + \lambda_h)^k |P| = 2^{k \log_2(\lambda_1 + \ldots + \lambda_h)}|P| = K^{r + \log_2{h}}|P|. \] 
Freiman's theorem \cite{fre73} says that, roughly speaking, sets with small doubling are contained in generalized arithmetic progressions with bounded dimension.  Using this line of reasoning, Schoen and Shkredov \cite[Theorem 6.2]{schshk16} proved that 
\[  |\lambda_1 \cdot A + \ldots + \lambda_h \cdot A | \le e^{O( \log_2^6 (2K) \log_2\log_2 (4K)) (h + \log_2 \sum_i |\lambda_i| ) }|A|. \]
This naturally leads us to ask a more precise version of Question~\ref{bukhconj}.
\begin{question}
If $|A+A| = K|A|$, then is
\[ |\lambda_1 \cdot A + \ldots + \lambda_h \cdot A| \le K^{O(h+ \ln \sum_i |\lambda_i|  ) } |A| ?\]
\end{question}

\section*{Acknowledgment}
We thank Boris Bukh and Giorgis Petridis for reading an earlier version of this paper and giving valuable comments. 
We also thank an anonymous referee for suggesting \cite{Tuza84} as well as other helpful comments.

\end{document}